\documentclass[12pt, reqno]{amsart}
\usepackage{latexsym}
\usepackage{mathtools}
\usepackage{amsmath}
\usepackage{amssymb}
\usepackage{mathrsfs}
\usepackage[abbrev]{amsrefs}
\usepackage{graphicx}
\usepackage[all]{xy}
\usepackage[normalem]{ulem}
\usepackage[utf8]{inputenc}
\usepackage{enumitem}

\newlist{altenumerate}{enumerate}{1}
\setlist*[altenumerate]{label=\textbf{(c\arabic*)}, resume=alt}

\newcommand{\id}{\ensuremath{\mathrm{id}}}
\newcommand{\B}{\ensuremath{\mathcal{B}}}
\newcommand{\Lop}{\operatorname{\rm{Lip}}}
\newcommand{\Mop}{\operatorname{\rm{D}}}
\newcommand{\KK}{\operatorname{\rm{K}}}
\newcommand{\F}{\ensuremath{\mathcal{F}}}
\newcommand{\FF}{\ensuremath{\mathbb{F}}}
\newcommand{\CC}{\ensuremath{\mathbb{C}}}
\newcommand{\II}{\ensuremath{\mathcal{I}}}
\newcommand{\PP}{\ensuremath{\mathcal{P}}}
\newcommand{\HH}{\ensuremath{\mathcal{H}}}
\newcommand{\Ind}{\ensuremath{\mathbf{1}}}

\newcommand{\Lip}{\ensuremath{\mathrm{Lip}}}
\newcommand{\Rea}{\ensuremath{\mathbb{R}}}
\newcommand{\Nat}{\ensuremath{\mathbb{N}}}
\newcommand{\MM}{\ensuremath{\mathcal{M}}}
\newcommand{\NN}{\ensuremath{\mathcal{N}}}
\newcommand{\ee}{\ensuremath{\mathbf{e}}}
\newcommand{\spn}{\operatorname{\mathrm{span}}\nolimits}

\newcommand{\supp}{\operatorname{\mathrm{supp}}\nolimits}

\allowdisplaybreaks

\newtheorem{Theorem}{Theorem}[section]
\newtheorem{Lemma}[Theorem]{Lemma}
\newtheorem{Proposition}[Theorem]{Proposition}
\newtheorem{Corollary}[Theorem]{Corollary}

\newtheorem{thmx}{Theorem}

\theoremstyle{remark}
\newtheorem{Remark}[Theorem]{Remark}
\newtheorem{Convention}[Theorem]{Convention}
\newtheorem{Definition}[Theorem]{Definition}
\newtheorem{Example}[Theorem]{Example}

\AtBeginDocument{
\def\MR#1{}
}

\subjclass[2010]{26A16; 46A16; 46B20; 46B80; 46B85}

\keywords{Metric space, Lipschitz algebra, Lipschitz free space}

\begin{document}

\title[Lipschitz-free spaces over unbounded spaces]{Lipschitz algebras and Lipschitz-free spaces over unbounded metric spaces}

\author[F. Albiac]{Fernando Albiac}
\address{Department of Mathematics, Statistics, and Computer Sciencies--InaMat2 \\
Universidad P\'ublica de Navarra\\
Campus de Arrosad\'{i}a\\
Pamplona\\
31006 Spain}
\email{fernando.albiac@unavarra.es}

\author[J. L. Ansorena]{Jos\'e L. Ansorena}
\address{Department of Mathematics and Computer Sciences\\
Universidad de La Rioja\\
Logro\~no\\
26004 Spain}
\email{joseluis.ansorena@unirioja.es}

\author[M. C\'uth]{Marek C\'uth}
\address{Faculty of Mathematics and Physics, Department of Mathematical Analysis\\
Charles University\\
186 75 Praha 8\\
Czech Republic}
\email{cuth@karlin.mff.cuni.cz}

\author[M. Doucha]{Michal Doucha}
\address{Institute of Mathematics\\
Czech Academy of Sciences\\
\v Zitn\'a 25\\
115 67 Praha 1\\
Czech Republic}
\email{doucha@math.cas.cz}

\begin{abstract}
We investigate a way to turn an arbitrary (usually, unbounded) metric space $\MM$ into a bounded metric space $\B$ in such a way that the corresponding Lipschitz-free spaces $\F(\MM)$ and $\F(\B)$ are isomorphic. The construction we provide is functorial in a weak sense and has the advantage of being explicit. Apart from its intrinsic theoretical interest, it has many applications in that it allows to transfer many arguments valid for Lipschitz-free spaces over bounded spaces to Lipschitz-free spaces over unbounded spaces. Furthermore, we show that with a slightly modified point-wise multiplication, the space $\Lip_0(\MM)$ of scalar-valued Lipschitz functions vanishing at zero over any (unbounded) pointed metric space is a Banach algebra with its canonical Lipschitz norm.
\end{abstract}

\thanks{F. Albiac acknowledges the support of the Spanish Ministry for Science and Innovation under Grant PID2019-107701GB-I00 for \emph{Operators, lattices, and structure of Banach spaces}. F. Albiac and J.~L. Ansorena acknowledge the support of the Spanish Ministry for Science, Innovation, and Universities under Grant PGC2018-095366-B-I00 for \emph{An\'alisis Vectorial, Multilineal y Aproximaci\'on}. M.~C\'uth has been supported by Charles University Research program No. UNCE/SCI/023. M. Doucha was supported by the GA\v{C}R project EXPRO 20-31529X and RVO: 67985840.}

\maketitle

\section{Introduction}
\noindent
Lipschitz spaces over metric spaces and their canonical preduals form by now a fundamental class of Banach spaces. By a Lipschitz space over a pointed metric space $(\MM,0)$ we understand the Banach space of all scalar-valued Lipschitz functions vanishing at $0$ with the minimal Lipschitz constant as a norm; further denoted by $\Lip_0(\MM)$. We recall that the canonical predual of $\Lip_0(\MM)$ is the Lipschitz-free space $\F(\MM)$, also known as the Arens-Eels space. The geometry of Lipschitz-free spaces is nowadays one of the most active fields of study within Banach space theory (see, e.g., \cites{AFGZ20, AACD20JFA, AP19, APP19, CKK19, GL18, N20, OO20, W18} for a non-exhaustive list of some recent developments).

Although Lipschitz-free spaces can be defined over any metric space, their theory can be developed more smoothly over bounded metric spaces. Indeed, the space $\Lip_0(\MM)$ is closed under pointwise multiplication if and only if $\MM$ is bounded, in which case $\Lip_0(\MM)$ becomes a topological algebra (even a Banach algebra after renorming). This additional algebraic structure makes these spaces more interesting and enriches their study. We refer the reader to the monograph \cite{WeaverBook2018} and to the papers \cites{CC11, CCD19, ChM19, R89, W94, W95} for some more information and recent advances on the subject. One does not face a priori such problems when dealing with Lipschitz-free spaces over unbounded metric spaces. However, a recurring pattern in the life of a Lipschitz-free space researcher is that proofs are often much simpler in the bounded case, where one does not need to deal with difficulties arising from the geometry of metric spaces at large scales.

In the literature we find classes of metric spaces whose Lipschitz-free space is isomorphic to the Lipschitz-free space over some bounded set. Perhaps the best-known example is the class of Banach spaces, whose Lipschitz free spaces are isomorphic to the Lipschitz free spaces over their unit balls (see \cite{Kaufmann2015}). However, in those cases the isomorphism is not explicit because it relies on Pe\l czy\'{n}ski's decomposition method. Our first main result in this article addresses this problem by showing that Lipschitz-free spaces, or more generally Lipschitz-free $p$-spaces for any $0<p\leq 1$, over bounded metric spaces are universal in the following sense.

\begin{thmx}\label{thm:A}
For an arbitrary metric space $\MM$ there is a \emph{bounded} metric space $\B(\MM)$ such that the Lipschitz free $p$-spaces $\F_p(\MM)$ and $\F_p(\B(\MM))$ are isomorphic, for $0<p\leq 1$.
Moreover, the association $\MM\mapsto\B(\MM)$ is a functor from the category of metric spaces with bi-Lipschitz maps as morphisms.
\end{thmx}

We emphasize that  the isomorphism between $\F(\MM)$ and $\F(\B(\MM))$ advertised in Theorem~\ref{thm:A} is given by an explicit simple formula and preserves many properties of isometric nature.
 Although we obtained our result independently, with hindsight, the isomorphism from the statement of Theorem~\ref{thm:A}   in the $p = 1$ case could also be deduced from \cite{WeaverBook2018}*{Theorem 2.20},  which establishes an isomorphism between 
$\Lip_0(\MM)$ and $\Lip_0(\B(\MM\setminus{0}))\simeq \Lip_0(\B(\MM))$.
Here we focus on  the Lipschitz-free space $\F(\B(\MM))$ instead of on its dual $\Lip_0(\B(\MM)$,  and  push in another direction the techniques derived from the reductions that the  metric space $\B(\MM)$ provides.  See Remark~\ref{rem:otherAuthors} for more details.

Our second main result is that $\Lip_0(\MM)$ becomes a Banach algebra when equipped with a simply defined product for any metric space $\mathcal M$, without needing to change the regular norm of the space $\Lip_0(\MM)$.

\begin{thmx}\label{thm:B}
For an arbitrary metric space $\MM$ there is an explicit modification of the pointwise multiplication on $\Lip_0(\MM)$ given by a simple formula which turns $\Lip_0(\MM)$ with its regular norm into a Banach algebra.
\end{thmx}

Both Theorems~\ref{thm:A} and \ref{thm:B} hold in the context of real as well as complex Banach spaces; Theorem~\ref{thm:A} holds even in the general setting of $p$-Banach spaces. Usually in the literature on the subject, it is considered the case of real spaces only, but in the context of Banach algebras the complex field  is more natural.

Let us now describe the contents of the paper in some detail. Section~\ref{Sec:2} is devoted to the proof of Theorem~\ref{thm:A}. We prove the isomorphism between $\F_p(\MM)$ and $\F_p(\B(\MM))$ with a quantitative estimate of $3^{1/p}$, and we have $d_{\B_p(\MM)}(x,0)\leq 1$ for $x\in \B_p(\MM)$. The functoriality of $\B$ is shown at the end of the section. In Section~\ref{Sec:3}, we compare the topology and geometry of the metric spaces $\MM$ and $\B(\MM)$, which will be helpful for applications.

In Section~\ref{subsec:Multiplication} we prove Theorem~\ref{thm:B}, that is, we define a product on $\Lip_0(\MM)$ so that $\Lip_0(\MM)$ becomes a Banach algebra when $\MM$ is an unbounded metric space, and we study the duality with respect to $\F(\MM)$. In Section~\ref{Sec:4} then use our construction to transfer well-known results about the Lipschitz algebra $\Lip_0(\mathcal B)$, for a bounded metric space $\mathcal B$, to $\Lip_0(\mathcal M)$ when $\MM$ is unbounded. For instance, we deal with $w^*$-closed ideals in $\Lip_0(\MM)$ (see Theorem~\ref{thm:weakStarClosedSubalgebras} and Theorem~\ref{thm:weakStarClosedIdeals}), with the fact that the algebraic structure of $\Lip_0(\MM)$ determines the linear topological structure of $\F(\MM)$ (see Proposition~\ref{p:algHomBdd} and Theorem~\ref{thm:surprisignlyStrongResult}), and we identify the spectrum of the Banach algebra $\Lip_0(\MM)$ (see Theorem~\ref{thm:w*EqualsNorm} and Theorem~\ref{thm:spectrum}).

Finally, in Section~\ref{Sec:6} we present a small selection of known results where the proofs for bounded metric spaces are much easier than for unbounded ones. The usefulness of our techniques has to be seen in that they allow an easy transfer of the simpler proofs to the unbounded case.

\subsection{Terminology}
Let $(\MM, d)$ and $(\NN,d')$ be metric spaces, and let $f\colon\MM\to\NN$ be Lipschitz. The Lipschitz norm of $f$ is given by
\[
\Lip(f) =\sup_{x\not=y} \frac{d'(f(x),f(y))}{d(x,y)}<\infty.
\]
If $(\MM,0)$ is a pointed metric space, the linear space $\Lip_0(\MM)$ with $\Lip(\cdot)$ as a norm is a Banach space (real or complex, depending on the scalar field $\FF=\Rea$ or $\FF = \CC$). The choice of the distinguished point $0\in\MM$ does not affect the linear isometric structure of $\Lip_0(\MM)$.

We say that $\MM$ and $\NN$ are \emph{Lipschitz isomorphic}, and we put $\MM\simeq_{\Lip} \NN$, if there is a bi-Lipschitz bijection from $\MM$ onto $\NN$.

Given a pointed metric space $(\MM,d,0)$, we can construct $\F(\MM)$, the (real or complex) Lipschitz-free space over $\MM$, and a natural isometric enbedding $\delta_\MM\colon\MM\to\F(\MM)$. It is well-known that $\F(\MM)$ is the canonical predual of $\Lip_0(\MM)$. By $\mu(f)$, $f(\mu)$ or $\langle \mu,f \rangle$ we denote the evaluation of the functional given by $f\in\Lip_0(\MM)$ on $\mu\in\F(\MM)$. The definition of Lipschitz-free spaces and their basic properties in the general setting of $p$-metric spaces are summarized in Subsection~\ref{subsec:pComplex}.


\section{Universality of Lipschitz free $p$-spaces over bounded metric spaces for $0<p\le 1$ }\label{Sec:2}
\noindent
The aim of this section is to build a bounded metric space $\B$ from an unbounded metric space $\MM$ and to prove that $\F(\MM)\simeq \F(\B)$. The main result here is Theorem~\ref{thm:boundedEquivToUnbounded}. Since all our results hold true for real and complex $p$-Banach spaces ($p\in(0,1]$) we decided to describe everything in this general setting and gather everything the reader needs to know about $p$-Banach spaces, Lipschitz-free $p$-spaces, and the complex-scalar case in Subsection~\ref{subsec:pComplex}. Those who are not interested in $p$-metric and $p$-Banach spaces can safely skip the next subsection and replace $p$ with $1$ in the sequel; just keep in mind that $1$-metric space means metric space.

\subsection{Remarks about the case $p<1$ and the complex scalar case}\label{subsec:pComplex}
Let $0<p\leq 1$. Recall that $(\MM,d)$ is a $p$-metric space if $(\MM,d^{p})$ is a metric space. In turn, a $p$-Banach space is a (real or complex) vector space $X$ equipped with a map $\|\cdot\|_X:X\to [0,\infty)$ that satisfies all the usual properties of the norm with the exception that the triangle inequality is replaced with
\[
\|\gamma_1+\gamma_2\|^p_X\leq \|\gamma_1\|^p_X + \|\gamma_2\|^p_X,\qquad \gamma_1,\gamma_2\in X,
\]
and, moreover, $(X,\|\cdot\|_X)$ is complete.

Given a pointed metric space $(\MM,d,0)$ one can construct the so-called \emph{Lipschitz free $p$-space over $\MM$} as the completion of the vector space $\PP(\MM):=\spn\{\delta(x)\colon x\in\MM\setminus\{0\}\}$ of all finitely supported scalar valued functions on $\MM\setminus\{0\}$, where $\delta(x)$ denotes the characteristic function of $\{x\}$ with the $p$-norm
\[
\left\Vert\sum_{i=1}^n a_i \delta(x_i)\right\Vert=\sup \left\Vert\sum_{i=1}^n a_j f(x_i)\right\Vert_X,
\]
the supremum being taken over all $p$-Banach spaces $(X,\|\cdot\|_X)$ and all choices of $1$-Lipschitz maps $f\colon\MM\rightarrow X$ with $f(0)=0$. In the case $p=1$, by the Hahn-Banach theorem it is enough to take the supremum over all $1$-Lipschitz maps in $\Lip_0(\MM)$. Alternatively, the norm on $\PP(\MM)$ can be defined using a Kantorovich-Rubinstein type-formula, for which we refer to \cite{AACD2018}.

Given $0<p<q\le 1$, there is a canonical norm-one linear map
\[
E_{p,q,\MM} \colon \F_p(\MM) \to \F_q(\MM).
\]
It is known \cite{AACD2018}*{Proposition 4.20} that the $q$-Banach envelope of $\F_p(\MM)$ is $\F_q(\MM)$ with $q$-envelope map $E_{p,q,\MM}$.
In particular the Banach envelope of $\F_p(\MM)$ is $\F(\MM)$. Thus,
the dual space $(\F_p(\MM))^*$ is isometrically isomorphic to the Banach space $\Lip_0(\MM)$ (see e.g. \cite{AACD2018}*{Corollary 4.23}). The duality is given $\langle \delta(x), f \rangle = f(x)$ for $x\in\MM$ and $f\in\Lip_0(\MM)$. By $\delta(0)$ we often denote the origin in $\PP(\MM)$. Thus, we easily infer the following well-known observation which we record for further reference.

\begin{Lemma}\label{lem:w*nets}
Let $(\MM,d,0)$ be a pointed metric space, $(f_i)_{i\in I}$ be a bounded net in $\Lip_0(\MM)$, and $f\in \Lip_0(\MM)$. Then $w^*$\textendash$\lim_i f_i=f$ if and only if $\lim_i f_i=f$ pointwise.
\end{Lemma}

\subsection{Construction of the bounded metric space in Theorem~\ref{thm:A}}\label{subsec:main}
We start our construction of a bounded metric space from an arbitrary metric space (in general unbounded) by introducing the main ingredients.

\begin{Definition}\label{def:Bp}
Let $\alpha\colon(0,\infty)\to(0,\infty)$ be a map and $0<p\le 1$. Given a metric space $(\MM,d)$ we set
\[
\B_p(\MM,\alpha)=\{0\}\cup\left\{\frac{\delta_\MM(x)}{\alpha(d(x,0))}\colon x\in \MM\setminus\{0\}\right\}
\]
and $\B(\MM,\alpha)=\B_1(\MM,\alpha)$. Once $\alpha$ is clear, we write $\B_p(\MM)$ and $\B(\MM)$ instead of $\B_p(\MM,\alpha)$ and $\B(\MM,\alpha)$, respectively.
\end{Definition}

\begin{Convention}
The map $\alpha:(0,\infty)\to(0,\infty)$ from Definition~\ref{def:Bp} will often be Lipschitz. In this case, we will denote by $\alpha(0)$ the limit $\lim_{t\to 0^+} \alpha(t)$.
\end{Convention}

Let $\alpha\colon(0,\infty)\to(0,\infty)$ be a function and $p\in(0,1]$. Given a pointed metric space $(\MM,d,0)$ we consider the maps $\zeta_{\alpha}\colon\MM\to[0,\infty)$ given by
\begin{equation}\label{eq:notation:a}
\zeta_\alpha(x):=\begin{cases} 0 & \text{ if } x=0,\\
\alpha(d(0,x)) & \text{ if } x\neq 0,
\end{cases}
\end{equation}
and $\mu_{\alpha}\colon\MM\to \F_p(\MM)$ given by
\begin{equation}\label{eq:notation:b}
\mu_\alpha(x):=\begin{cases} 0 & \text{ if } x=0,\\
\displaystyle \frac{\delta_\MM(x)}{\zeta_\alpha(x)} & \text{ if } x\neq 0.
\end{cases}
\end{equation}
We will write $\zeta$ instead of $\zeta_\alpha$ and $\mu$ instead of $\mu_\alpha$ once $\alpha$ is clear. We have $\B_p(\MM,\alpha)= \mu(\MM)$.

We also need to introduce the constants
\[
\Mop(\alpha):=\sup_{t>0}\frac{t}{\alpha(t)}
\]
and
\[ 
\KK(\alpha)=\Lip(\alpha)\, \Mop(\alpha).
\]
Note that if $\Mop(\alpha)<\infty$ then $\lim_{t\to \infty} \alpha(t)=\infty$ and so
\[
\KK(\alpha)\ge \sup_{t>1} \frac{t}{\alpha(t)}\frac{|\alpha(t)-\alpha(1)|}{t-1}\ge 1.
\]
Lemma~\ref{lem:compbis} below provides estimates for the distance in the $p$-metric space $\B_p(\MM,\alpha)$ in terms of $\Mop(\alpha)$, $\Lop(\alpha)$, and the distance $d$ in $\MM$. Prior to proving it we give a couple of auxiliary results. With the convention $0/0=0$ we set
\begin{equation}\label{eq:notation:c}
d_\alpha(x,y)=\frac{d(x,y)}{\max\{\zeta_\alpha(x),\zeta_\alpha(y)\}}, \quad x,y\in\MM.
\end{equation}
\begin{Lemma}\label{lem:comp}
Let $\alpha\colon(0,\infty)\to(0,\infty)$ be a map, $(\MM,d,0)$ be a pointed metric space, and $0<p\le 1$. Let $\zeta$ and $\mu$ be as in \eqref{eq:notation:a} and \eqref{eq:notation:b}, respectively.
\begin{enumerate}[label={(\roman*)}, leftmargin=*, widest=iii]
\item\label{lem:comp:1} If $\Mop(\alpha)<\infty$, then $\B_p(\MM,\alpha)$ is a bounded subset of $\F_p(\MM)$. In fact, if $f\in \Lip_0(\MM)$ is given by $f(z)=d(0,z)$ for all $z\in\MM$, \[
\Vert \mu(x)\Vert_{\F_p(\MM)} =\mu(x)(f)=\frac{d(0,x)}{\zeta(x)} \le \Mop(\alpha), \quad x\in\MM\setminus\{0\}.
\]
\item\label{lem:comp:2} If $\alpha$ is Lipschitz, then $\zeta$ is Lipschitz on $\MM\setminus\{0\}$. Quantitatively,
\[
|\zeta(x)-\zeta(y)|\le \Lop(\alpha)\,d(x,y), \quad x, y\in \MM\setminus\{0\}.
\]
\end{enumerate}
\end{Lemma}

\begin{proof}
\ref{lem:comp:1} is clear from the definition, and \ref{lem:comp:2} follows by combining the Lipschitz condition with the triangle inequality.
\end{proof}

\begin{Lemma}\label{lem:37}
Let $\alpha\colon(0,\infty)\to(0,\infty)$ be a map and $(\MM,d,0)$ be a pointed metric space. Let $\zeta$, $\mu$ and $d_\alpha$ be as in \eqref{eq:notation:a}, \eqref{eq:notation:b}, and \eqref{eq:notation:c} respectively. Then, for every $y\in\MM\setminus\{0\}$ there are $f$, $g\in \Lip_0(\MM)$ with $\Lip(f) \le 1$ such that:
\begin{enumerate}[label={(\roman*)}, leftmargin=*, widest=ii]
\item\label{lem:37:1} $|(\mu(x)-\mu(y))(f)|\ge d_\alpha(x,y)$ whenever $x\in\MM\setminus\{0\}$ satisfies $\zeta(x)\ge \zeta(y)$; and
\item\label{lem:37:2} for every $x\in\MM\setminus\{0\}$ either
$|(\mu(x)-\mu(y))(f)|\ge d_\alpha(x,y)$ or $|(\mu(x)-\mu(y))(g)|\ge d_\alpha(x,y)$.
\end{enumerate}
\end{Lemma}

\begin{proof}
For $z\in\MM$ put $f(z)= d(z,y)-d(0,y)$. Then
\[
(\mu(x)-\mu(y))(f) = \frac{d(x,y)-d(0,y)}{\zeta(x)}+\frac{d(0,y)}{\zeta(y)}
\]
for all $x\in\MM\setminus\{0\}$.
We infer that $(\mu(x)-\mu(y))(f)\ge d_\alpha(x,y)$ in the case when $\zeta(y)\le \zeta(x)$ or $d(x,y)\ge d(0,y)$.
This proves \ref{lem:37:1}. To conclude the proof of \ref{lem:37:2}, we set
\[
g_1=\frac{\zeta}{\zeta(y)} \wedge 1 \text{,}\quad g_2=|f|\wedge d(0,y) \text{, and}\quad\; g=g_1g_2.
\]

Since $g_1$ is bounded and $g_2$ is continuous with $g_2(0)=0$, we deduce that $g$ is continuous at $0$ with $g(0)=0$.
By Lemma~\ref{lem:comp}~\ref{lem:comp:2}, $g_1|_{\MM\setminus\{0\}}$ is Lipschitz and so is $g_2|_{\MM\setminus\{0\}}$. Taking into account that $g_1$ and $g_2$ are bounded we infer that $g|_{\MM\setminus\{0\}}$ is Lipschitz. Hence, $g\in\Lip_0(\MM)$. Let $x\in\MM\setminus\{0\}$ be so that $\zeta(x)<\zeta(y)$ and $d(x,y)<d(0,y)$. Then $g_1(x)=\zeta(x)/\zeta(y)$, $g_1(y)=1$, $g_2(x)=d(0,y)-d(x,y)$, and $g_2(y)=d(0,y)$. Therefore,
\[
(\mu(x)-\mu(y))(g)=\frac{\zeta(x)(d(0,y)-d(x,y))}{ \zeta(y)\zeta(x)} -\frac{ d(0,y)}{\zeta(y)}=-d_\alpha(x,y).\qedhere
\]
\end{proof}

\begin{Lemma}\label{lem:compbis}
Let $\alpha\colon(0,\infty)\to(0,\infty)$ be a Lipschitz map with $\Mop(\alpha)<\infty$, $(\MM,d,0)$ be a pointed metric space, and $0<p\le 1$. Let $\zeta$ and $\mu$ be as in \eqref{eq:notation:a}, \eqref{eq:notation:b}, and \eqref{eq:notation:c} respectively. Then
\[
d_\alpha(x,y)\le \left\| \mu(x) - \mu(y) \right\|_{\F_p(\MM)}\leq (1+\KK^p(\alpha))^{1/p} d_\alpha(x,y), \quad x,y\in\MM.
\]
\end{Lemma}
\begin{proof}
It suffices to consider the case when $x$, $y\in\MM\setminus\{0\}$ and $\zeta(x)\ge \zeta(y)$. The left hand-side inequality follows from Lemma~\ref{lem:37}~\ref{lem:37:1}. Using Lemma~\ref{lem:comp}, the expression
\[
\mu(x)-\mu(y)=\frac{1}{\zeta(x)}(\delta_\MM(x)-\delta_\MM(y))+\frac{\zeta(y)-\zeta(x)}{\zeta(x)}\mu(y),
\]
yields
\begin{align*}
\Vert\mu(x)-\mu(y)\Vert^p_{\F_p(\MM)}
&\le\frac{d^{p}(x,y)}{\zeta^p(x)}+\Mop^p(\alpha) \Lop^p(\alpha)\frac{d^{p}(x,y)}{\zeta^{\, p}(x)}\\
&=(1+\KK^p(\alpha)) d_\alpha^p(x,y).\qedhere
\end{align*}
\end{proof}

Proposition~\ref{prop:BL} below exhibits that, up to Lipschitz isomorphism, it suffices to consider $\B_p(\MM,\alpha)$ in the case when $p=1$ and $\alpha$ is either the identity map on $(0,\infty)$, which we will denote by $\alpha^{(0)}$, or the map $\alpha^{(1)}=1+\alpha^{(0)}$.

\begin{Remark}\label{rem:otherAuthors}
The metric space $\B:=\B(\MM,\alpha^{(0)})$ is in a sense natural and has previously appeared in the work of other authors. 
 Weaver (see \cite{WeaverBook2018}*{Definition~{2.17} and Lemma~{2.18}})
 considered $\B$ when investigating the lattice structure of $\Lip_0(\MM)$. To be more specific, he identified the \emph{complete lattice spectrum} of $\Lip_0(\MM)$ as the set $\B\setminus \{0\}$ (\cite{WeaverBook2018}*{Theorem 6.22}) and proved that there is a natural order-preserving surjective isomorphism between $\Lip_0(\MM)$ and $\Lip(\B\setminus\{0\})$
  (\cite{WeaverBook2018}*{Theorem 6.23}).
From a completely different  approach, Aliaga et al.\  proved that the set of extreme points of the positive cone of the unit ball of  $\F(\MM)$  coincides with our space $\B$ (see \cite{APP19}*{Theorem 3.8}).
\end{Remark}

Given two real-valued functions $f$ and $g$ we write $f\approx g$ to denote that there are constants $C,D>0$ such that $Cf\leq g\leq Df$.

\begin{Lemma}\label{lem:equivalpha}
Let $\alpha\colon(0,\infty)\to(0,\infty)$ be a Lipschitz map with $\Mop(\alpha)<\infty$. If $\alpha(0)=0$, then $\alpha\approx\alpha^{(0)}$; and, if $\alpha(0)>0$, then $\alpha\approx\alpha^{(1)}$.
\end{Lemma}

\begin{proof}
If $\alpha(0)=0$ we have $\alpha(t)\approx t$ for $0<t\le 1$. If $\alpha(0)>0$ we have $\alpha(t)\approx 1$ for $0<t\le 1$. Since $\alpha(t)\approx t$ for $t\ge 1$ we are done.
\end{proof}

\begin{Proposition}\label{prop:BL}
Let $(\MM,d,0)$ be a pointed metric space, let $0<p\le 1$, and let $\alpha\colon(0,\infty)\to(0,\infty)$ be a Lipschitz map with $\Mop(\alpha)<\infty$. Then:
\begin{enumerate}[label={(\roman*)}, leftmargin=*, widest=iii]
\item\label{prop:BL:1} $\B_p(\MM,\alpha)\simeq_{\Lip} \B(\MM,\alpha)$. To be precise, the identity map is a bi-Lipchitz map with distortion at most
$
(1+\KK^p(\alpha))^{1/p}.
$
\item\label{prop:BL:2} Either $\alpha(0)=0$, in which case $\B(\MM,\alpha) \simeq_{\Lip} \B(\MM,\alpha^{(0)})$, or $\alpha(0)>0$, in which case $\B(\MM,\alpha)\simeq_{\Lip}\B(\MM,\alpha^{(1)})$.
\item\label{prop:BL:3} If $0$ is an isolated point, $\B(\MM,\alpha^{(0)})\simeq_{\Lip} \B(\MM,\alpha^{(1)})$.
\end{enumerate}
\end{Proposition}

\begin{proof}
Let $\zeta$ and $d_\alpha$ be as in \eqref{eq:notation:a} and \eqref{eq:notation:c}, respectively.
Let $\mu_q$ denote the function $\mu$ defined as in \eqref{eq:notation:b} corresponding to an index $q\in\{p,1\}$. Let $x$, $y\in\MM$. By lemma~\ref{lem:compbis},
\begin{equation*}
(1+\KK^p(\alpha))^{-1/p} \Vert\mu_p(x)-\mu_p(y)\Vert_{\F_p(\MM)}\le d_\alpha(x,y) \le \Vert\mu_1(x)-\mu_1(y)\Vert_{\F(\MM)}.
\end{equation*}
Since $\Vert m\Vert_{\F(\MM)}\le \Vert m \Vert_{\F_p(\MM)}$ for every molecule $m$, \ref{prop:BL:1} holds. Moreover,
\begin{equation}\label{eq:equiv}
\Vert\mu_1(x)-\mu_1(y)\Vert_{\F(\MM)}\approx d_\alpha(x,y), \quad x,y\in\MM.
\end{equation}

\ref{prop:BL:2} follows from combining Lemma~\ref{lem:equivalpha} with \eqref{eq:equiv}. If $d:=d(0,\MM\setminus\{0\})>0$, then $d_\alpha$ only depends on the values of $\alpha$ in the interval $[d,\infty]$. Therefore, since $\alpha^{(0)}(t) \approx \alpha^{(1)}(t)$ for $t\ge d$, $d_{\alpha^{(0)}}(x,y)\approx d_{\alpha^{(1)}}(x,y)$ for all $x$, $y\in \MM$. Applying \eqref{eq:equiv} yields \ref{prop:BL:3}.
\end{proof}

\begin{Theorem}\label{thm:boundedEquivToUnbounded}
Let $(\MM,d)$ be a metric space. There exists a bounded metric space $\B$ with $\F_p(\MM)\simeq \F_p(\B)$ for all $p\in(0,1]$. To be more specific, if $\alpha\colon (0,\infty)\to(0,\infty)$ is a Lipschitz function with $\Mop(\alpha) < \infty$, and $\zeta$ and $\mu$ are as in \eqref{eq:notation:a} and \eqref{eq:notation:b} respectively, then the linear operator $P_\alpha^F\colon \F_p(\MM)\to \F_p(\B(\MM,\alpha))$ defined by
\[
P_\alpha^F(\delta_\MM(x)) = \zeta(x)\delta_\B(\mu(x)), \quad x\in \MM\setminus\{0\},
\]
induces an isomorphism with inverse $Q_\alpha^F\colon \F_p(\B(\MM,\alpha))\to \F_p(\MM)$ given by
\[
Q_\alpha^F(\delta_\B(\mu(x)))=\mu(x), \quad x\in \MM\setminus\{0\}.
\]
Moreover, $\|P_\alpha^F\| \, \|Q_\alpha^F\|\leq C$, where
\[
C=\begin{cases} 1+ 2 \KK(\alpha) & \text{ if } p=1, \\ (1+\KK^p(\alpha) )^{1/p} (1+2\KK^p(\alpha))^{1/p} & \text{ if } p<1.\end{cases}
\]
\end{Theorem}

\begin{proof}
By Proposition~\ref{prop:BL}~\ref{prop:BL:1}, we can replace $\B(\MM,\alpha)$ with $\B:=\B_p(\MM,\alpha)$. By Lemma~\ref{lem:comp}~(i), $\B$ is bounded. For $\zeta$ and $\mu$ as in \eqref{eq:notation:a} and \eqref{eq:notation:b} respectively, consider the map $f\colon\MM\to \F_p(\B)$ defined by the formula
\[
f(x)=\zeta(x) \, \delta_\B(\mu(x)), \quad x\in\MM.
\]
Let us verify that $f$ is $(1+2\KK^p(\alpha))^{1/p}$-Lipschitz. Let $x$, $y\in \MM$, and set $D=\| f(x) - f (y)\|_{\F_p(\B)}$. Without loss of generality we may assume that $\zeta(x)\geq\zeta(y)$. The expansion
\[
f(x)-f(y)= \zeta(x) \, ( \delta_\B(\mu(x))- \delta_\B(\mu(y))) + (\zeta(x) -\zeta(y)) \, \delta_\B(\mu(y)),
\]
combined with Lemmas~\ref{lem:comp} and ~\ref{lem:compbis} gives
\begin{align*}
D^p
&\le \zeta^p(x) \Vert \mu(x)-\mu(y)\Vert_{\F_p(\MM)}^p +|\zeta(x)-\zeta(y)|^p \Vert \mu(y)\Vert_{\F_p(\MM)}^p\\
&\le (1+\KK^p(\alpha)) d^p (x,y) + \Lip^p(\alpha) d^p(x,y) \Mop^p(\alpha) \\
& =(1+2\KK^p(\alpha))d^p(x,y).
\end{align*}
Then, \cite{AACD2018}*{Theorem 4.5} yields a bounded linear operator $P_\alpha^F\colon\F_p(\MM)\to \F_p(\B)$ such that
\[
P_\alpha^F(\delta_{\MM}(x)) = \zeta(x)\, \delta_\B(\mu(x)), \quad x\in\MM.
\]
Further, since the inclusion of $\B$ into $\F_p(\MM)$ is $1$-Lipschitz, appealing again to \cite{AACD2018}*{Theorem 4.5}, we infer that there is a norm-one linear operator $Q_\alpha^F\colon\F_p(\B)\to \F_p(\MM)$ with
\[
Q_\alpha^F(\delta_\B(\mu(x))) = \mu(x), \quad x\in\MM.
\]
It is clear from the definition that $P_\alpha^F$ restricted to $\spn\{\delta_\MM(x) \colon x\in\MM\setminus\{0\}\}$ is a linear bijection onto $V=\spn\{\delta_\B(x) \colon x\in\B\setminus\{0\}\}$ with inverse $Q_\alpha^F|_V$. Therefore, $P_\alpha^F$ and $Q_\alpha^F$ are inverse isomorphisms of each other.
\end{proof}

We close this section with a couple of examples that we work out for simplicity in the case of real Banach spaces, that is, $\FF=\Rea$.

\begin{Example}
Consider the metric space $\Nat\cup\{0\}$ endowed with the Euclidean distance. The Lipschitz-free space $\F(\Nat\cup\{0\})$ is isometric to $\ell_1$ via the map $\delta(n)\mapsto\sum_{j=1}^n \ee_j$, where $\ee_j$ denotes the $j$th unit vector, $n\in\Nat$. Hence, if we set
\[
\B=\{0\}\cup \left\{ \frac{1}{n}\sum_{j=1}^n \ee_j \colon n\in\Nat\right\} \subset\ell_1,
\]
we have $\F_p(\Nat\cup\{0\})\simeq \F_p(\B)$ for every $0<p\le 1$.
\end{Example}

\begin{Example}
Consider the metric space $\Rea^+=[0,\infty)$ endowed with the Euclidean distance. The Lipschitz-free space $\F(\Rea^+)$ is isometric to $L_1(\Rea^+)$. Namely, the map $\delta(x)\mapsto\chi_{[0,x)}$ extends to an isometry. Hence, if we consider
\begin{align*}
\B_0&=\{0\}\cup \left\{ \frac{1}{x} \chi_{[0,x)} \colon x>0\right\},\\
\B_1& =\left\{ \frac{1}{x+1} \chi_{[0,x)} \colon x\ge 0\right\}
\end{align*}
as subsets of $L_1(\Rea^+)$, we have $\F_p(\Rea^+)\simeq \F_p(\B_0)\simeq \F_p(\B_1)$ for every $0<p\le 1$.
\end{Example}

\subsection{The functorial character of our construction}
Recall that the construction of Lipschitz-free spaces is functorial in nature. Indeed, there is a functor $\F$ from the category of pointed metric spaces equipped with Lipschitz maps between them preserving the base points to the category of Banach spaces. To each pointed metric space $\MM$, the functor $\mathcal F$ associates a Lipschitz-free space $\F(\MM)$ and to each Lipschitz map $f\colon \MM\rightarrow \NN$, with $f(0)=0$, it associates its linear extension between $\F(\MM)$ and $\F(\NN)$. Moreover, $\F$ is a left-adjoint functor to the forgetfull functor from the category of Banach spaces to the category of pointed metric spaces.

It is very natural to investigate to which extent the construction of a bounded metric space $\B(\MM)$ from a metric space $\MM$ is functorial too. We shall next show that it is functorial in a weak sense - namely, when we restrict the class of morphisms to bi-Lipschitz zero-preserving maps only (not necessarily bijections).

For any zero-preserving Lipschitz map $f\colon\MM\rightarrow\NN$ between two pointed metric spaces and any $0<p\le 1$ there is a canonical linear extension $ L_f\colon\F_p(\MM)\rightarrow\F_p(\NN)$. If $P_{p,\MM}$ and $P_{p,\NN}$ denote the linear isomorphisms provided by Theorem~\ref{thm:boundedEquivToUnbounded},
we obtain a  bounded linear map $\B(f)\colon \F_p(\B(\MM))\rightarrow\F_p(\B(\NN))$ making the diagram

\begin{equation*}
\begin{gathered}
\xymatrix{
\F_p(\MM) \ar@{->}[r]^-{L_f} \ar@{->}[d]^-{P_\MM} & \F_p(\NN) \ar@{->}[d]^-{P_\NN} \\
\F_p(\B(\MM)) \ar@{-->}[r]^-{\B(f)} & \F_p(\B(\NN))
}
\end{gathered}
\end{equation*}
commute. A straightforward computation yields
\[
\B(f)(\delta(\mu(x)))=\frac{\zeta(f(x))}{\zeta(x)} \delta(\mu(f(x))),\qquad x\in\MM\setminus\{0\}.
\]
Therefore $\B(f)$ is not
a linear extension of a Lipschitz map between $\B(\MM)$ and $\B(\NN)$ unless
$\zeta(f(x))=\zeta(x)$ for all $x\in\MM\setminus\{0\}$,
as it happens for example when $f$ is an isometric embedding, and the construction is carried out with the same function $\alpha$ in both metric spaces.

However, when the category is restricted to the class of bi-Lipschitz morphisms, then the construction is functorial in the regular sense.  In the  proposition below, $\id$ denotes the identity map on a given metric space.
\begin{Proposition}\label{prop:functorial}
Let $\alpha\colon (0,\infty)\rightarrow (0,\infty)$ be a Lipschitz map. For any bi-Lipchitz zero-preserving map $f\colon\MM\to\NN$ between pointed metric spaces let us consider $\B(f)\colon\B(\MM)\to \B(\NN)$ defined as
\[\B(f)(\mu(x)):=\mu(f(x)),\quad x\in \MM,\]
where $\mu$ is as in \eqref{eq:notation:b}.
Then for every bi-Lipchitz zero-preserving maps $f$, $g\colon\MM\to\NN$ between pointed metric spaces we have:
\begin{enumerate}[label=(\roman*), leftmargin=*,widest=ii]
\item $\B(\id)=\id$ and $\B(f\circ g) = \B(f)\circ \B(g)$, and
\item $\B(f)$ is a bi-Lipchitz zero-preserving map.
\end{enumerate}
\end{Proposition}

\begin{proof}
We only verify that $\B(f)$ is a Lipschitz map between $\B(\MM)$ and $\B(\NN)$. By symmetry, it will follow that it is bi-Lipschitz, and it is clear that $\B(\id)=\id$ and $\B(f\circ g)=\B(f)\circ\B(g)$.

Fix $x,y\in \MM$. Let us estimate the ratio
\[
\frac{\|\B(f)(\mu(x))-\B(f)(\mu(y))\|}{\|\mu(x)-\mu(y)\|}=\frac{\|\mu(f(x))-\mu(f(y))\|}{\|\mu(x)-\mu(y)\|}.
\] By Lemma~\ref{lem:compbis},
\[
d_\alpha(x,y)\le \left\| \mu(x) - \mu(y) \right\|_{\F_p(\MM)}\leq (1+\KK^p(\alpha))^{1/p} d_\alpha(x,y),
\]
and analogously for $f(x)$ and $f(y)$.
It follows that
\[
\frac{\|\mu(f(x))-\mu(f(y))\|}{\|\mu(x)-\mu(y)\|} \leq \frac{(1+\KK^p(\alpha))^{1/p} d_\alpha(f(x),f(y))}{d_\alpha(x,y)},
\]
so it is enough to provide a uniform bound for
\[
\frac{d_\alpha(f(x),f(y))}{d_\alpha(x,y)}=A(x,y) B(x,y),
\]
where 
\[
A(x,y)=\frac{d(f(x),f(y))}{d(x,y)} \text{ and } B(x,y)=\frac{\max\{\zeta(x),\zeta(y)\}}{\max\{\zeta(f(x)),\zeta(f(y))\}}.
\]
Since
$A(x,y)\le \Lip(f)$,
it suffices to obtain a uniform bound for
$B(x,y)$. To that end, without loss of generality
we assume  that
$\zeta(y)\le\zeta(x)$.
By Lemma~\ref{lem:equivalpha} we may assume that $\alpha=\alpha^{(0)}$ or $\alpha=\alpha^{(1)}$; in particular $\alpha$ is increasing. Then,
\[
B(x,y)
\leq \frac{\zeta(x)}{\zeta(f(x))}\leq \frac{\alpha(d(x,0))}{\alpha(\Lip(f^{-1})d(x,0))},
\]
 and considering separately the cases $\alpha=\alpha^{(0)}$ and $\alpha=\alpha^{(1)}$ it is routine to check   that the right-hand side term is uniformly bounded.
Note that the last inequality is the only place where we used that
 the inverse of $f$
is Lipschitz.
\end{proof}

\begin{Remark}\label{rmk:Lj}
It is known that if $\MM$ is a metric space and $\NN\subset \MM$, the canonical  linearization $L_j\colon\F_p(\NN)\to\F_p(\MM)$ of the inclusion map $j\colon \NN\to\MM$
need not be an isometry for $p<1$. The question of whether it is an isomorphism is open as of today (see \cite{AACD2018}*{Section 6}).  Since $\B(j)=L_j$,
an interesting consequence of Proposition~\ref{prop:functorial} is that the question reduces to bounded metric spaces; indeed,   if the answer were positive for bounded metric spaces $\MM$ then the answer would be positive also for unbounded ones.
\end{Remark}

\begin{Remark}
It follows directly from the construction that if $(\MM,d,0)$ is a metric space and $0\in\NN\subset\MM$, then $\B(\NN)\subset \B(\MM)$ via the canonical isometry. It also follows from the proof of Proposition~\ref{prop:functorial} that if $f\colon\MM\to\NN$ is a zero-preserving isometry and $\alpha = \alpha^{(0)}$, then $\B(f)\colon\B(\MM)\to\B(\NN)$ is 
 an isometry
as well.
\end{Remark}

\begin{Remark}
Our construction also has a functorial behavior with respect to the family of spaces that arise if $p$ is allowed to take values in interval $(0,1]$: for any pointed metric space $\MM$ and any $0<p<q\le 1$ the diagram
\begin{equation}\label{eq:pqdiagram}
\begin{gathered}
\xymatrix{
\F_p(\MM) \ar@{->}[r]^-{P_{p,\MM}} \ar@{->}[d]^-{E_{p,q,\MM}} & \F_p(\B(\MM)) \ar@{->}[d]^-{E_{p,q,\B(\MM)}} \\
\F_q(\MM) \ar@{->}[r]^-{P_{q,\MM}} & \F_q(\B(\MM)).
}
\end{gathered}
\end{equation}
commutes. Another open problem from \cite{AACD2018} is whether the envelope map $E_{p,1,\MM}\colon\F_p(\MM)\to\F(\MM)$ is one-to-one for any metric space $\MM$ and any $0<p<1$. By \eqref{eq:pqdiagram}, to address this question it suffices to consider the case when $\MM$ is bounded. We recall that $E_{p,1,\MM}$  is always one-to-one on $\PP(\MM)$, but this does not guarantee its injectivity on its completion $\F_p(\MM)$ (see \cite{AAW2020} for a discussion about the injectivity of envelope maps).
\end{Remark}

\section{The metric space $\MM$ versus the bounded metric space $\B(\MM,\alpha)$}\label{Sec:3}
\noindent
In this section we compare the Lipschitz structures and the topological structures of the unbounded metric space $\MM$ on one hand, and the bounded metric space $\B(\MM,\alpha)$ built from $\MM$ and a Lipschitz function $\alpha$ with $\Mop(\alpha)<\infty$ on the other. We prove that $\MM$ and $\B(\MM,\alpha)$ are topologically homeomorphic whenever $0\in\MM$ is an isolated point.

The assumption that $0\in\MM$ is isolated is not so strong. Indeed, if $\MM$ does not have isolated points we add an isolated point $*$ to our space $\MM$ so that $\F_p(\MM\cup\{*\})\simeq \F_p(\MM)$ for every $0<p\leq 1$ (see e.g. \cite{AACD2020}*{Lemma 2.8}). This way we obtain a homeomorphism between $\MM$ and $\B:=\B(\MM\cup\{*\},\alpha)\setminus \{0\}$ such that $\F_p(\MM)\simeq \F_p(\B)$.

Everything that follows in this section holds for both real and complex spaces, because we are using exclusively the properties of Lipchitz-free spaces summarized in Subsection~\ref{subsec:pComplex}.

\begin{Lemma}\label{lem:10}
Let $(\MM,d,0)$ be a metric space, and let $\alpha$ be a map from $(0,\infty)$ into $(0,\infty)$.

\begin{enumerate}[label={(\roman*)}, leftmargin=*, widest=viii]
\item\label{lem:10:2}Let $d>0$ and $x\in\MM\setminus\{0\}$ be such that $d(0,x)\ge d$. If $\alpha$ is Lipschitz then
\[
\Vert \mu(x)\Vert_{\F(\MM)} \ge\min\left\{ \frac{1}{\Lip(\alpha)} ,\frac{d}{\alpha(d)} \right\}.
\]
\item\label{lem:10:3}Suppose that $\alpha$ is Lipschitz, that $\Mop(\alpha)<\infty$, and that either $\alpha(0)>0$ or $0$ is an isolated point of $\MM$. Then there is a constant $C<\infty$ such that
\[
\Vert \mu(x)\Vert_{\F(\MM)} \le C d(0,x), \quad x\in\MM\setminus\{0\}.
\]
In fact, we can choose $\inf_{x\not=0} \alpha(d(0,x))=1/C$.
\item\label{lem:10:4}Suppose that $\alpha$ is Lipschitz and $\alpha(0)=0$. Then
\[
\Vert \mu(x)\Vert_{\F(\MM)} \ge \frac{1}{\Lip(\alpha)}, \quad x\in\MM\setminus\{0\}.
\]
\end{enumerate}
\end{Lemma}

\begin{proof}
Note that $\Vert \mu(x)\Vert_{\F(\MM)}=d(0,x)/\alpha(d(0,x))$ so \ref{lem:10:3} easily holds.
To prove \ref{lem:10:2} we set $u=d(0,x)-d$, so that
\[
\Vert \mu(x)\Vert_{\F(\MM)}
= \frac{u+d}{\alpha(u+d)}
\ge \frac{u+d}{\Lip(\alpha) u+\alpha(d)}
\ge\min\left\{ \frac{1}{\Lip(\alpha)} ,\frac{d}{\alpha(d)} \right\}.
\]

Finally, under the assumptions in \ref{lem:10:4}, we have $\alpha(t)\le \Lip(\alpha) t$ for all $t>0$, and the result is a consequence of this inequality.
\end{proof}

\begin{Proposition}\label{lem:propertiesOfTransformation}
Let $\alpha\colon(0,\infty)\to(0,\infty)$ be a Lipschitz map with $\Mop(\alpha)<\infty$ and let $(\MM,d,0)$ be a pointed metric space. Set $\B=\B(\MM,\alpha)$. Let $\mu$ be as in \eqref{eq:notation:b}. Then:
\begin{enumerate}[label={(\roman*)}, leftmargin=*, widest=viii]
\item\label{PoT:1} $\mu$ is one-to-one.
\item\label{PoT:2} $\mu|_{\MM_{[r,R]}}$ is bi-Lipschitz for all choices of $0<r<R<\infty$, where
\[
\MM_{[r,R]}=\{x\in \MM \colon r\le d(0,x) \le R \}.
\]
\item\label{PoT:7} If $0$ is an isolated point of $\MM$ or $\alpha(0)>0$, $\mu$ is Lipschitz, and $\mu|_{\MM_{[0,R]}}$ is bi-Lipschitz for every $R<\infty$.
\item\label{PoT:3} $\mu|_{\MM\setminus\{0\}}$ is a topological homeomorphism onto $\B\setminus\{0\}$.
\item\label{PoT:12} If $0\in K\subset\MM$, $K$ is closed in $\MM$ if and only if $\mu(K)$ is closed in $\B$.
\item\label{PoT:10} If $0$ is an isolated point of $\MM$ or $\alpha(0)>0$, $\mu$ is a topological homeomorphism.
\item\label{PoT:4} Given a net $(x_i)_{i\in I}$ in $\MM\setminus\{0\}$ the following are equivalent:
\begin{enumerate}[label={(\alph*)}, leftmargin=*, widest=b]
\item\label{PoT:4:a} $\lim_i \mu(x_i)=0$ in norm.
\item\label{PoT:4:b} $\lim_i \mu(x_i)=0$ weakly.
\item\label{PoT:4:c} $\alpha(0)>0$ and $\lim_i x_i=0$.
\end{enumerate}
\item\label{PoT:6} $0$ is an isolated point of $\B$ if and only if either $\alpha(0)=0$ or $0$ is an isolated point of $\MM$.
\item\label{PoT:11} The inverse of $\mu$ is continuous.
\item\label{PoT:5} A sequence $(x_n)_{n=1}^\infty$ in $\MM\setminus\{0\}$ is Cauchy if and only if either it converges to $0$ or $(\mu(x_n))_{n=1}^\infty$ is Cauchy in $\B$.
\item\label{PoT:8} If $0$ is an isolated point of $\MM$ or $\alpha(0)>0$, $\MM$ is complete if and only if $\B$ is complete.
\item\label{PoT:9} The norm and weak topologies coincide on $\B \subset \F(\MM)$.
\end{enumerate}
\end{Proposition}
\begin{proof}
Let $d_\alpha$ be as in \eqref{eq:notation:c}.
\ref{PoT:1} is clear, because $\{\delta_\MM(x)\colon x\not=0\}$ is a linearly independent set. \ref{PoT:2} and \ref{PoT:7} follow from Lemma~\ref{lem:compbis}. In turn, \ref{PoT:3} follows easily from \ref{PoT:2}, \ref{PoT:12} is a consequence of \ref{PoT:3}, and \ref{PoT:10} follows from \ref{PoT:7}.

The equivalence between \ref{PoT:4:a} and \ref{PoT:4:b} in \ref{PoT:4} is a consequence of Lemma~\ref{lem:comp}~\ref{lem:comp:1}. Let $(x_i)_{i\in I}$ be a net in $\MM\setminus\{0\}$. We infer from Lemma~\ref{lem:10}~\ref{lem:10:3} that $\lim_i \mu(x_i)=0$ whenever $\lim_i x_i=0$ and $\alpha(0)>0$, from Lemma~\ref{lem:10}~\ref{lem:10:2} that $(\mu(x_i))_{n=1}^\infty$ does not converge to $0$ whenever $(x_i)_{i\in I}$ does not converge to $0$, and from Lemma~\ref{lem:10}~\ref{lem:10:4} that $(\mu(x_i))_{i\in I}$ does not converge to $0$ whenever $\alpha(0)=0$. Thus, \ref{PoT:4:a} and \ref{PoT:4:c} in \ref{PoT:4} are equivalent.

\ref{PoT:6} is a consequence \ref{PoT:4}, and \ref{PoT:11} follows from \ref{PoT:3} and \ref{PoT:4}. We now prove \ref{PoT:5}. Let $(x_n)_{n=1}^\infty$ be a sequence in $\MM\setminus\{0\}$. Suppose that it is Cauchy and does not converge to $0$. Then,
$
\inf_n d(0,x_n)>0.
$
Applying Lemma~\ref{lem:compbis} gives that $(\mu(x_n))_{n=1}^\infty$ is a Cauchy sequence. Suppose that $(\mu(x_n))_{n=1}^\infty$ is a Cauchy sequence and assume by contradiction that $(x_n)_{n=1}^\infty$ is not Cauchy. Then, using Lemma~\ref{lem:compbis},
\[\lim_{k\to\infty} \sup_{m,n\geq k}d_\alpha(x_n,x_m) = 0 < \lim_{k\to\infty} \sup_{m,n\geq k} d(x_n,x_m),\] which easily implies $\sup_n \alpha(d(0,x_n))=\infty$. Passing to a subsequence we can assume that $(\alpha(d(0,x_n)))_{n=1}^\infty$ increases to $\infty$. Then, if $m\le n$,
\[
d_\alpha(x_n,x_m) =
\frac{d(x_n,x_m)} { \alpha(d(0,x_n)) } \ge \frac{ d(0,x_n)-d(0,x_m)}{\alpha(d(0,x_n))}.
\]
Thus, by Lemma~\ref{lem:equivalpha}, for every $m\in\Nat$ we have $\liminf_n d_\alpha(x_n,x_m) > 0$ and so applying Lemma~\ref{lem:compbis} we obtain
\[
\liminf_n \Vert \mu(x_n)-\mu(x_m)\Vert>0, \quad m\in\Nat.
\]
This absurdity concludes the proof of \ref{PoT:5}.

\ref{PoT:8} follows from \ref{PoT:5} and \ref{PoT:10}. By \ref{PoT:4}, in order to prove \ref{PoT:9} it suffices to show that the norm and weak topologies coincide on $\B\setminus\{0\}$. To that end, taking into accout Lemma~\ref{lem:compbis}, it suffices to show that for every $y\in\MM\setminus\{0\}$ and every $\varepsilon>0$ there is a finite family $(f_i)_{i=1}^n$ in $\Lip_0(\MM)$ and $\delta>0$ such that $d_\alpha(x,y)\le\varepsilon$ whenever $x\in\MM\setminus\{0\}$ satisfies $|(\mu(x)-\mu(y))(f_i)|\le \delta$ for all $1\le i \le n$. But this fact is immediate from Lemma~\ref{lem:37}~\ref{lem:37:2}.
\end{proof}

Roughly speaking, the following corollary of Proposition~\ref{lem:propertiesOfTransformation}~\ref{PoT:7} shows that applying our construction to an already  bounded metric space barely alters the original space.
\begin{Corollary}
Let $\alpha$ be a Lipschitz map with $\Mop(\alpha)<\infty$ and let $(\MM,d,0)$ be a pointed metric space. Suppose that $\MM$ is bounded and that either $0$ is an isolated point of $\MM$ or $\alpha(0)>0$.
Then $\MM\simeq_{\Lip}\B(\MM,\alpha)$.
\end{Corollary}

It is obvious that if two metric spaces $\MM$ and $\NN$ are bi-Lipschitz equivalent, then the corresponding Lipschitz free $p$-spaces $\F_p(\MM)$ and $\F_p(\NN)$ are linearly isomorphic for all $0<p\leq 1$. It is well-known that the converse does not hold, i.e., a linear isomorphism between $\F_p(\MM)$ and $\F_p(\NN)$ does not imply the bi-Lipschitz equivalence between $\MM$ and $\NN$.
 Theorem~\ref{thm:boundedEquivToUnbounded} provides examples of this situation.
With the tools obtained so far, we next show how the bi-Lipschitz equivalence between $\B(\MM,\alpha)$ and $\B(\NN,\alpha)$ stands in this comparison.

\begin{Theorem}
Let $\MM$ and $\NN$ be pointed metric spaces, let $\alpha\colon(0,\infty)\rightarrow (0,\infty)$ be a Lipschitz function with $D(\alpha)<\infty$, and let $0<p\leq 1$. Then the following implications hold: 
\[
\MM\simeq_{\Lip}\NN \Rightarrow \B(\MM,\alpha)\simeq_{\Lip}\B(\NN,\alpha)\Rightarrow \F_p(\MM)\simeq \F_p(\NN).
\]
Moreover, none of the converse implications is true even if the metric spaces $\MM$ and $\NN$ are complete.
\end{Theorem}
\begin{proof}
 The implications are a consequence of Proposition~\ref{prop:functorial} and Theorem~\ref{thm:boundedEquivToUnbounded}.
Let us
show that there exist metric spaces $\MM$ and $\NN$ with $\F(\MM)\simeq \F(\NN)$ but such that $\B(\MM,\alpha)$ and $\B(\NN,\alpha)$ are not bi-Lipschitz. Pick any compact metric space $K$ and a non-compact metric space $X$ such that $\F_p(X)\simeq \F_p(K)$ (e.g., by \cite{AACD2018}, it suffices to put $K=[0,1]$ and $X=[0,\infty)$). In particular, by attaching isolated points to both these spaces, there are a non-compact metric space $\MM$ with isolated $0$ and a compact metric space $\NN$ with isolated $0$ such that $\F_p(\MM)\simeq \F_p(\NN)$. By Proposition~\ref{lem:propertiesOfTransformation}~\ref{PoT:10}, $\B(\MM,\alpha)$ and $\MM$ are homeomorphic. The same is true for $\B(\NN,\alpha)$ and $\NN$. Therefore $\B(\MM,\alpha)$ and $\B(\NN,\alpha)$ cannot be bi-Lipschitz equivalent.

Finally, we find $\MM$ and $\NN$ such that $\MM$ and $\NN$ are not bi-Lipschitz equivalent, while $\B(\MM,\alpha)$ and $\B(\NN,\alpha)$ are. Let $\MM$ be a countable bounded uniformly separated metric space. By Proposition~\ref{lem:propertiesOfTransformation} \ref{PoT:2}, $\B(\MM,\alpha)\simeq_{\Lip} \MM$. Therefore $\B(\MM,\alpha)$ also is countable, bounded and uniformly separated.
 Consider $\NN=\Nat\cup\{0\}$ endowed with the metric
\[
d(m,n)=|2^m-2^n|, \quad m,n\in \Nat\cup\{0\}.
\]
For $ m$, $n\in \Nat\cup\{0\}$ with $m\not=n$ we have
\[
d(m,n)\approx 2^{\max\{m,n\}} \approx \max\{ \alpha(d(n,0)) , \alpha(d(m,0))\}.
\]
Thus, an application of Lemma~\ref{lem:compbis} shows that $\B(\NN,\alpha)$ is countable, bounded and uniformly separated. Therefore, $\B(\MM,\alpha)$ and $\B(\NN,\alpha)$ are bi-Lipschitz equivalent. On the other hand, since $\MM$ is bounded and $\NN$ is not, these spaces are not bi-Lipschitz equivalent.
\end{proof}

In Section~\ref{Sec:4}, we shall see how the Lipschitz equivalence between $\B(\MM,\alpha)$ and $\B(\NN,\alpha)$ is tightly related to the algebraic isomorphism between the algebras $\Lip_0(\MM)$ and $\Lip_0(\NN)$ (see Theorem~\ref{thm:surprisignlyStrongResult}).


\section{$\Lip_0(\MM)$ as a Banach algebra for $\MM$ unbounded metric space}\label{subsec:Multiplication}
\noindent
In this section we examine the (real or complex) Banach space $\Lip_0(\MM)$ as an algebra over an arbitrary metric space $\MM$. Recall that we have $\F(\MM)^*\equiv \Lip_0(\MM)$ (the fact that this is the case even for the complex field is explained in Subsection~\ref{subsec:pComplex}).
In the case when the pointed metric space $(\MM,d,0)$ is bounded, $\Lip_0(\MM)$ is a Banach algebra in the weak sense, that is, there is a constant $C>1$ such that
\[
\Lip(f g)\le C\Lip (f) \Lip(g), \quad f,g\in \Lip_0(\MM).
\]
In fact, we can choose $C=2\max_{x\in\MM} d(0,x)$. These are called \emph{Gelfand algebras} in \cite{WeaverBook2018}.
Note that the Banach algebra law requires that the above estimate holds with $C=1$. As it is discussed in \cite{WeaverBook2018}*{Chapter 7}, when $\MM$ is bounded it is possible to define a submultiplicative equivalent norm on $\Lip_0(\MM)$ so that $\Lip_0(\MM)$ becomes a Banach algebra. As Weaver explains, the drawbacks of changing the natural norm are the loss of the lattice structure of the unital ball of $\Lip_0(\MM)$ as well as the breakdown of several isometric results (see \cite{WeaverBook2018}*{Chapter 7} for a discussion on this topic).

Here we keep the standard Lipschitz norm on $\Lip_0(\MM)$. However, we redefine the product by a simple natural formula so that the resulting algebra is a Banach algebra for any (even unbounded) metric space $\MM$. In doing so, some of the results of isometric nature are preserved.

In the case when $\MM$ is unbounbed we shall use Theorem~\ref{thm:boundedEquivToUnbounded} as a vehicle to define a multiplication on $\Lip_0(\MM)$. Given a Lipschitz map $\alpha\colon(0,\infty)\to(0,\infty)$ with $\Mop(\alpha)<\infty$, let $Q_\alpha^F\colon\F(\B(\MM))\rightarrow \F(\MM)$ and $P_\alpha^F\colon \F(\MM)\rightarrow \F(\B(\MM))$ be as in Theorem~\ref{thm:boundedEquivToUnbounded}. $Q_\alpha^F$ and $P_\alpha^F$ are linear isomorphisms inverse of each other. Let $Q_\alpha^L\colon\Lip_0(\MM) \to \Lip_0(\B(\MM))$ and $P_\alpha^L\colon\Lip_0(\B(\MM)) \to \Lip_0(\MM)$ be their respectivel dual maps via the canonical isometries between Lipschitz spaces and the duals of Lipschitz-free spaces. Since pointwise multiplication is a well-defined operation on $\Lip_0(\B(\MM))$, the operation
\[
P_\alpha^L( Q_\alpha^L(f) Q_\alpha^L(g)), \quad f,g\in \Lip_0(\MM).
\]
is well-defined on $\Lip_0(\MM)$. Nevertheless, let us define a multiplication by a more transparent formula and show that it is equivalent to what is displayed above.

\begin{Definition} For $f,g\in\Lip_0(\MM)$ and $x\in\MM$ we define a new operation $\odot_\alpha$ on $\Lip_0(\MM)$
as
\[
f\odot_\alpha g(x):= \frac{f(x)g(x)}{\zeta(x)},
\]
where $\zeta$ is as in \eqref{eq:notation:a}.
\end{Definition}

A routine computation yields
\begin{equation}\label{eq:T*}
P_\alpha^L(h)(x)=\zeta(x) h(\mu(x)), \quad h\in \Lip_0(\B(\MM)),
\end{equation}
and
\begin{equation}\label{eq:S*}
Q_\alpha^L(f)(\mu(x))=\frac{f(x)}{\zeta(x)}, \quad f\in \Lip_0(\MM).
\end{equation}
As the alert reader might have noticed, expression \eqref{eq:S*} connects the operator $Q_\alpha^L$  with the work of Weaver  (see \cite{WeaverBook2018}*{Theorem~{2.20}}).

Thus, for $x\in\MM$, 
\begin{equation}
f\odot_\alpha g(x)=P_\alpha^L( Q_\alpha^L(f) Q_\alpha^L(g))(x),
\end{equation}
which immediately gives that $f\odot_\alpha g\in\Lip_0(\MM)$. Moreover, using that for Lipschitz functions on a bounded metric space $\NN$ we have the estimate
\[
\Lip(f'g')\leq 2\mathrm{diam}(\NN)\Lip(f')\Lip(g'),\]
combined with the fact that by Theorem~\ref{thm:boundedEquivToUnbounded},
\[
\|Q^L_\alpha\|\|P^L_\alpha\|\leq 1+2K(\alpha),
\] we get
\[
\Lip(f\odot_\alpha g)\leq 2\Mop(\alpha)(1+2\KK(\alpha))\Lip(f)\Lip(g).
\]

The following result is aimed at improving the constant $2\Mop(\alpha)(1+2\KK(\alpha))$ in the last inequality.

\begin{Lemma}\label{lem:lipConstantMultiplication}
Let $(\MM,d,0)$ be an unbounded pointed metric space and let $\alpha\colon(0,\infty)\to(0,\infty)$ be a Lipschitz map with $\Mop(\alpha)<\infty$. Then
\[
\Lip( f\odot_\alpha g) \leq \Mop(\alpha)(\KK(\alpha)+2) \Lip(f) \Lip(g)
\]
for all $f,g\in\Lip_0(\MM)$.
\end{Lemma}

\begin{proof}
Let $x$, $y\in\MM$. If $0\notin\{x,y\}$, using Lemma~\ref{lem:comp}, we estimate $E:=|f\otimes_\alpha g(x)- f\otimes_\alpha g(y)|$ by
\begin{align*}
& \left|\frac{(f(x)-f(y))g(x)}{\zeta(x)}\right| + \left|\frac{(g(x)-g(y))f(y)}{\zeta(y)}\right|+ \left|\frac{f(y)g(x)(\zeta(y) - \zeta(x))}{\zeta(x)\zeta(y)}\right|\\
&\leq \Lip(f) \Lip(g) d(x,y)\left( \frac{d(0,x)}{\zeta(x)} + \frac{d(0,y)}{\zeta(y)}+ \Lip(\alpha )\frac{d(0,x)d(0,y)}{\zeta(x) \zeta(y)} \right)\\
&\leq (2\Mop(\alpha) + \Mop^2(\alpha) \Lip(\alpha) ) \Lip(f) \Lip(g) d(x,y).
\end{align*}

If $y=0$ and $x\not=0$,
\[
E \leq \frac{ \Lip(f) \Lip(g) d(0,x)d(0,x)}{\zeta(x)}
\leq \Mop(\alpha) \Lip(f) \Lip(g) d(y,x).\qedhere
\]
\end{proof}

Note that the constant obtained in Lemma~\ref{lem:lipConstantMultiplication} is optimal. Indeed, for $\MM = \Rea^+$, $\alpha(t)=3t$, $f(t) = 1-|t-1|$ and $g(t)=-f(t)$, we have
\begin{align*}
\Lip(f\odot_\alpha g) & \geq \lim_{\varepsilon\to 0^+}\frac{|(f\odot_\alpha g)(1) - (f\odot_\alpha g)(1+\varepsilon)|}{\varepsilon}\\ & = (f\odot_\alpha g)'_+(1)\\
& = 1\\ & = \Mop(\alpha)(\KK(\alpha)+2) \Lip(f) \Lip(g).
\end{align*}

Our next result is a ready consequence of our construction. Recall that given a metric space $\MM$ we have a natural ordering on $\Lip_0(\MM)$, namely, $f\geq 0$ if $f(x)\geq 0$ for every $x\in\MM$. Then $\Lip_0(\MM)$ is an ordered Banach algebra, i.e., if we put
\[
\Lip_0^+(\MM)=\{ f\in \Lip_0(\MM) \colon f\ge 0\},\] the sets
\[
\Rea^+\cdot \Lip_0^+(\MM), \quad \Lip_0^+(\MM)+\Lip_0^+(\MM), \quad \Lip_0^+(\MM) \cdot \Lip_0^+(\MM)
\]
are contained in the positive cone $\Lip_0^+(\MM)$.

Let $\NN$ be another metric space. A map $T\colon\Lip_0(\MM)\to \Lip_0(\NN)$ is said to be \emph{normal} if whenever $(f_i)_{i\in I}$ is a norm-bounded net in $\Lip_0(\MM)$ which increases pointwise towards $f\in \Lip_0(\MM)$, then the net $(T(f_i))_{i\in I}$ increases pointwise towards $T(f)$. Any normal operator is positive.

\begin{Proposition}\label{cor:canonicalMorphism}(cf.\  \cite{WeaverBook2018}*{Theorem 6.23}). 
Let $(\MM,d)$ be an unbounded metric space, and let $\alpha\colon (0,\infty)\to(0,\infty)$ be a Lipschitz function with $\Mop(\alpha) < \infty$. Then, the map $Q_\alpha^L\colon\Lip_0(\MM)\to \Lip_0(\B)$ defined as in \eqref{eq:S*} is a $w^*$-$w^*$ continuous normal isomorphism with $Q_\alpha^L(f\odot_\alpha g) = Q_\alpha^L(f)Q_\alpha^L(g)$ for all $f$ and $g\in\Lip_0(\MM)$. Its inverse is the operator $P_\alpha^L$ defined in \eqref{eq:T*}.
\end{Proposition}

Let us now investigate the existence of a unit for the Banach algebra $\Lip_0(\MM)$.
\begin{Lemma}\label{lem:unit}Let $(\MM,d)$ be an unbounded metric space, and let $\alpha$ be a Lipschitz function with $\Mop(\alpha) < \infty$. Consider $\Lip_0(\MM)$ endowed with the multiplication $\odot_\alpha$.
Let $\zeta$ be defined as in \eqref{eq:notation:a}. The following are equivalent:
\begin{enumerate}[label={(\roman*)}, leftmargin=*, widest=iii]
\item\label{lem:unit:1} $\Lip_0(\MM)$ has a unit, in which case such a unit is $\zeta$.
\item\label{lem:unit:2} $\zeta$ is a Lipschitz function.
\item\label{lem:unit:3} $\zeta$ is continuous at zero.
\item\label{lem:unit:4} Either $0$ is an isolated point or $\alpha(0)=0$.
\end{enumerate}
\end{Lemma}

\begin{proof}
Suppose that $e\in \Lip_0(\MM)$ is a unit. Then $e(x)=\zeta(x)$ for all $x\in \NN:=\cup_{f\in\Lip_0(f)} (\MM \setminus f^{-1}(0))$. Since $\NN=\MM\setminus\{0\}$, $e=\zeta$ and so $\zeta$ is continuous.
This proves \ref{lem:unit:1} $\Rightarrow$ \ref{lem:unit:2}, and the converse is clear. The equivalence between \ref{lem:unit:2} and \ref{lem:unit:3} follows from Lemma~\ref{lem:comp}~\ref{lem:comp:2}, and the equivalence between \ref{lem:unit:3} and \ref{lem:unit:4} follows from
Lemma~\ref{lem:equivalpha}.
\end{proof}

We point out here that there are maps $\alpha\colon(0,\infty) \to (0,\infty)$ satistying $\Mop(\alpha)(\KK(\alpha)+2)\le 1$, so that, in light of Lemma~\ref{lem:lipConstantMultiplication}, the space $\Lip_0(\MM)$ becomes a Banach algebra with the multiplication $\odot_\alpha$; take for instance $\alpha(t)=3t$ for $t>0$. However, there is no choice of $\alpha$ for which $\Lip_0(\MM)$ becomes a unital Banach algebra (that is, a Banach algebra with a norm-one unit). Indeed, if $\Lip_0(\MM)$ had a unit $e$ with $\Lip(e)=1$, by Lemma~\ref{lem:unit} it would have to be $e=\zeta$. Then we would have $\alpha(t)\le t$ for all $t\in\{ (d(0,x) \colon x\in \MM\}$. Therefore $\Mop(\alpha)\ge 1$ and so $\Mop(\alpha)(\KK(\alpha)+2)\ge 2$.

\begin{Remark}
Given a Banach algebra $X$ there is an explicit formula for an equivalent norm $|\cdot|$ on $X$ such that $(X,|\cdot|)$ becomes a unital Banach algebra (see \cite{DalesBook}*{Proposition 2.1.9}). However, as we mentioned above, this renorming leads to the loss of the lattice structure of the unit ball of $\Lip_0(\MM)$.
\end{Remark}

\begin{Remark}
Using Lemma~\ref{lem:w*nets} the multiplication $\odot_\alpha$ on $\Lip_0(\MM)$ is $w^*$-$w^*$ separately continuous, or more precisely it is $w^*$-$w^*$ separately continuous on bounded sets and then we use the Banach-Dieudonn\'e theorem (see, e.g., \cite{AlbiacKalton2016}*{Appendix G.8}). So in the terminology of \cite{RundeBook}*{Chapter 5}, $(\Lip_0(\MM),\F(\MM))$ is a dual Banach algebra. Moreover, in the complex scalar case it is easy to see that the involution on $\Lip_0(\MM)$ defined by $f\mapsto \overline{f}$ shows that $\Lip_0(\MM)$ is also a Banach $*$-algebra (see \cite{DalesBook}*{Definition 3.1.1} for the corresponding definition).
\end{Remark}

\section{From Lipschitz algebras over unbounded metric spaces to Lipschitz algebras over bounded metric spaces}\label{Sec:4}
\noindent
In this section we deal only with Banach spaces over the real field. In \cite{WeaverBook2018}*{Chapter 7} there are many results about the algebra $\Lip_0(\B)$ for bounded metric spaces $\B$, which now easily transfer to the ordered Banach algebra $\Lip_0(\MM)$ over an unbounded metric space $\mathcal M$. Note however that in \cite{WeaverBook2018} the case of complex Banach spaces is not dealt with, which is the reason why we deal here with real Banach spaces only.

Troughout this section $(\MM,d,0)$ will be a pointed metric space, $\alpha\colon (0,\infty)\to(0,\infty)$ will be a map with $\Mop(\alpha)(\KK(\alpha)+2)\le 1$, and $\mu$ and $\zeta$ will be the functions defined in \eqref{eq:notation:a} and \eqref{eq:notation:b} respectively. We will not a priori assume that $\Lip_0(\MM)$ has a natural unit, i.e., we do not impose $\zeta$ to be continuous at $0$.

\begin{Definition}
Let $(\NN,d,0)$ be a pointed metric space and let $Y\subset \Lip_0(\NN)$ be a subalgebra.
\begin{enumerate}[label={(\roman*)}, leftmargin=*, widest=ii]
\item We say that $Y$ is \emph{order complete}, if it is stable under pointwise convergence of norm-bounded increasing nets.
\item We say that $Y$ is a \emph{linear complete sublattice} if it is closed under taking the supremum of an arbitrary, possibly infinite, set of functions that are uniformly bounded in the Lipschitz norm.
\end{enumerate}
\end{Definition}

\begin{Lemma}\label{lem:fg-cunit}
Suppose that $\Lip_0(\MM)$ has a natural unit. Let $f$, $g\in \Lip_0(\MM)$, and $c\in[0,\infty)$. Then, if $1_\MM$ denotes the unit of $\Lip(\MM)$, and $Q_\alpha^L$ is as in \eqref{eq:S*}, $Q_\alpha^L(f)\vee (Q_\alpha^L(g)-c \,1_\MM)\in \Lip_0(\MM)$ and, if $P_\alpha^L$ is as in \eqref{eq:T*},
\[
P_\alpha^L( Q_\alpha^L(f)\vee (Q_\alpha^L(g)-c \,1_\MM) )= f\vee (g-c\, \zeta).
\]
\end{Lemma}

\begin{proof}
It is a routine checking.
\end{proof}

\begin{Theorem}\label{thm:weakStarClosedSubalgebras}
A subalgebra $Y\subset \Lip_0(\MM)$ is order complete if and only if it is $w^*$-closed. Moreover, if $Y$ is order complete then:
\begin{enumerate}[label={(\roman*)}, leftmargin=*, widest=ii]
\item\label{thm:wscs:1} $Y$ is linear complete sublattice.
\item\label{thm:wscs:2} If $\Lip_0(\MM)$ has a natural unit, we have $f\vee (g-c\, \eta)\in Y$ for all $f$, $g\in Y$ and all $c\geq 0$.
\end{enumerate}
\end{Theorem}
\begin{proof}
\ref{thm:wscs:1} follows from combining \cite{WeaverBook2018}*{Lemma 7.6, Corollary 7.7 and Lemma 7.9} with Lemma~\ref{lem:w*nets} and Proposition~\ref{cor:canonicalMorphism}, which asserts that $Q_\alpha^L$ is a $w^*$-$w^*$-homeomorphism that preserves multiplication and ordering. To see \ref{thm:wscs:2}, we need to take into account also Lemma~\ref{lem:fg-cunit}.
\end{proof}

\begin{Definition}
A subalgebra $Y$ of a Banach algebra $X$ is said to be an \emph{ideal} if $xy$, $yx\in I$ for every $x\in X$ and $y \in Y$, and it is said to be a \emph{complete band} if it is a linear complete sublattice, and for every $y_1$, $y_2\in Y$ and $x\in X$ with $y_1\leq x\leq y_2$ we have $x\in Y$.
\end{Definition}

\begin{Definition}
Given a metric space $\NN$ and a subspace $Y$ of $\Lip_0(\NN)$, define the \emph{hull} of $Y$ as the closed set
\[
\HH_\NN(Y):=\{x\in\NN\colon f(x)=0\text{ for all }f\in Y\}.
\]
Given a closed set $K\subset\MM$, we put
\[
\II_\MM(K):=\{f\in\Lip_0(\NN)\colon f(x)=0\text{ for all }x\in K\}.
\]
\end{Definition}
The indices in $\II_\MM$, resp. $\HH_\NN$, will be omitted if the metric space is clear from the context.

Note that $\HH(\II(K)) = K$ for every closed set $K\subset\NN$ as shown by the Lipschitz map $x\mapsto d(x,K)$.

\begin{Lemma}\label{lem:lHK}
Let $P_\alpha^L\colon\Lip_0(\B)\to \Lip_0(\MM)$ be as in \eqref{eq:T*}, where $\B=\B(\MM,\alpha)$. Then,
\[
P_\alpha^L(\II_\B(\HH_\B(Y))) = \II_\MM(\HH_\MM(P_\alpha^L(Y)))
\]
for any subspace $Y$ of $\Lip_0(\B)$.
\end{Lemma}

\begin{proof}
It is clear from Proposition~\ref{cor:canonicalMorphism} and the definition of $P_\alpha^L$.
\end{proof}

\begin{Theorem}\label{thm:weakStarClosedIdeals}
Let $Y\subset \Lip_0(\MM)$ be an ideal. Then the following conditions are equivalent.
\begin{enumerate}[label={(\roman*)}, leftmargin=*, widest=iii]
\item\label{wscs1} $Y$ is order complete.
\item\label{wscs2} $Y$ is $w^*$-closed.
\item\label{wscs3} $Y$ is a complete band.
\item\label{wscs4} $Y = \II(\HH(Y))$.
\item\label{wscs5} There exists a closed set $K\subset \MM$ such that $Y = \II(K)$.
\end{enumerate}
\end{Theorem}
\begin{proof}
The equivalence between \ref{wscs1} and \ref{wscs2} follows from Theorem~\ref{thm:weakStarClosedSubalgebras}. \ref{wscs2}$\Rightarrow$\ref{wscs3} follows from \cite{WeaverBook2018}*{Lemma 7.16}, where it is proved for bounded metric spaces, and from Proposition~\ref{cor:canonicalMorphism}. The implication \ref{wscs3}$\Rightarrow$\ref{wscs4} follows from \cite{WeaverBook2018}*{Proof of Theorem 6.19} and Lemma~\ref{lem:lHK}. Finally, \ref{wscs4}$\Rightarrow$ \ref{wscs5} is obvious, and \ref{wscs5}$\Rightarrow$ \ref{wscs1} is immediate from the definition.
\end{proof}

\begin{Theorem}\label{thm:keyPartInTheIntersectionTheorem}
If $Y\subset \Lip_0(\MM)$ is an ideal, then $\overline{Y}^{w^*} = \II(\HH(Y))$.
\end{Theorem}
\begin{proof}
The result follows by combining the corresponding result for bounded metric spaces proved in \cite{AP19}*{Proposition 3.2} with Proposition~\ref{cor:canonicalMorphism} and Lemma~\ref{lem:lHK}.
\end{proof}

Let $\NN$ be another metric space. A map $T\colon\Lip_0(\MM)\to \Lip_0(\NN)$ is an \emph{algebra homomorphism with respect to $\alpha$} if it is linear and preserves the multiplication $\odot_\alpha$. If $\alpha$ is known, we simply write \emph{algebra homomorphism}
\begin{Proposition}\label{p:algHomBdd}
If $T\colon\Lip_0(\MM)\to \Lip_0(\NN)$ is an algebra homomorphism then it is bounded (i.e., continuous) and preserves ordering.
\end{Proposition}
\begin{proof}
It follows from \cite{WeaverBook2018}*{Proposition 7.29}, where it is proved for bounded metric spaces, and Proposition~\ref{cor:canonicalMorphism}.
\end{proof}

By $\Delta(\MM)$ we denote the set of all nonzero algebra homomorphisms from $\Lip_0(\MM)$ into $\Rea$, where $\Rea$ is of course equipped with its standard addition and multiplication. By $\Delta_0(\MM)$ we denote the set $\Delta(\MM)\cup \{0\}$, where $0$ is the zero map. By $\Delta^*(\MM)$ and $\Delta_0^*(\MM)$ we denote the set of all members of $\Delta(\MM)$ and $\Delta_0(\MM)$ respectively that are $w^*$-continuous.

A straightforward computation gives that, if we regard $\F(\MM)$ as a subspace of $(\Lip_0(\MM))^*$, then $\mu(x)\in\Delta(\MM)$ for all $x\in\MM\setminus\{0\}$.
The following lemma identifies the set $\Delta_0^*(\MM)$.

\begin{Lemma}\label{lem:normalCharacters}
Let $\mu$ be as in \eqref{eq:notation:b}. For $\omega\in\Delta(\MM)$ the following are equivalent:
\begin{enumerate}[label={(\roman*)}, leftmargin=*, widest=iii]
\item\label{nC:1} $\omega$ is normal.
\item\label{nC:2} $\omega$ is $w^*$-continuous.
\item\label{nC:3} $\omega =\mu(x)$ for some $x\in\MM\setminus\{0\}$.
\end{enumerate}
In particular, we have $\Delta_0^*(\MM) = \B(\MM,\alpha)$.
\end{Lemma}

\begin{proof}
Just combine \cite{WeaverBook2018}*{Lemma 7.22} with Proposition~\ref{cor:canonicalMorphism}.
\end{proof}

The following is an analogue to \cite{WeaverBook2018}*{Theorem 7.23}.

\begin{Theorem}\label{thm:algHomo}
Let $\NN$ be another metric space, and assume that the multiplication $\odot_\alpha$ admits a natural unit in both $\MM$ and in $\NN$. Let $T\colon\Lip_0(\MM)\to\Rea^\NN$ be a mapping. The following are equivalent:
\begin{enumerate}[label={(\roman*)}, leftmargin=*, widest=ii]
\item\label{thm:algH:1} $T\colon \Lip_0(\MM)\to\Lip_0(\NN)$ is a unital normal algebra homomorphism.
\item\label{thm:algH:2} There exists $g\colon\NN\setminus\{0\}\to\MM\setminus\{0\}$ such that $T=D_g$ and $D_g(\Lip_0(\MM))\subset \Lip_0(\NN)$, where $D_g\colon \Lip_0(\MM)\to \Rea^\NN$ is the map given by
\[
D_g(f)(x)=\frac{\alpha(d(0,x))}{\alpha(d(0,g(x)))}f(g(x))=\frac{\zeta(x)}{\zeta(g(x))}f(g(x)), \quad x\in\NN\setminus\{0\}.
\]
\end{enumerate}
\end{Theorem}

\begin{proof}
The implication \ref{thm:algH:2}$\Rightarrow$\ref{thm:algH:1} is straightforward. Let us assume that \ref{thm:algH:1} holds. Let $\zeta'$ and $\mu'$ be the maps defined in \eqref{eq:notation:a} corresponding to the metric space $\NN$. Fix $x\in\NN\setminus\{0\}$. Since $\mu'(x)\in\Delta(\NN)$ and $T(\zeta)=\zeta'$ (because $T$ is unital and $\zeta$, $\zeta'$ are the units), we obtain $\mu'(x)\circ T\neq 0$ and so $\mu'(x)\circ T\in \Delta(\MM)$. Since $\mu'(x)$ is normal, $\mu'(x)\circ T$ is normal. Applying Lemma~\ref{lem:normalCharacters} yields $g(x)\in \MM\setminus\{0\}$ such that $\mu'(x)\circ T=\mu(g(x))$, i.e.,
\[
\delta_\NN(x)\circ T=\frac{\zeta'(x)}{\zeta(g(x))}\delta_\MM(g(x)).\qedhere
\]
\end{proof}

The following is an analogue of \cite{WeaverBook2018}*{Theorem 7.26}. It seems this is actually interesting even outside of the framework of Banach algebras.

\begin{Theorem}\label{thm:w*EqualsNorm}
The norm and weak$^*$ topologies of $(\Lip_0(\MM))^*$ coincide on $\Delta_0^*(\MM)$. Moreover, if $\Lip_0(\MM)$ has a natural unit, then $0$ is an isolated point of $\Delta_0^*(\MM)$.
\end{Theorem}

\begin{proof} The first part follows from combining Proposition~\ref{lem:propertiesOfTransformation}~\ref{PoT:9} with Lemma~\ref{lem:normalCharacters}. Combining Lemma~\ref{lem:unit} with Proposition~\ref{lem:propertiesOfTransformation}~\ref{PoT:6} yields the `moreover' part.\end{proof}

\begin{Theorem}\label{thm:spectrum}
Suppose $\Lip_0(\MM)$ has a natural unit. Then the spectrum of the Banach algebra $\Lip_0(\MM)$ verifies
\[
\Delta_0(\MM) = \overline{\B(\MM,\alpha)}^{w^*}\subset \Lip_0(\MM)^{**},
\]
and
\[
\Delta(\MM) = \overline{\B(\MM,\alpha)\setminus\{0\}}^{w^*}\subset \Lip_0(\MM)^{**}.
\]
\end{Theorem}
\begin{proof}
By Lemma~\ref{lem:normalCharacters} and Theorem~\ref{thm:w*EqualsNorm}, $\B:=\B(\MM,\alpha) = \Delta_0^*(\MM)$, and $0$ is an isolated point of $\Delta_0^*(\MM)$.

By \cite{WeaverBook2018}*{Lemma 7.28} the set $\Delta_0^*(\B)$ is dense in the $w^*$-closed set $\Delta_0(\B)\subset\Lip_0(\B)^{**}$. Let $Q_\alpha^L$ be as in \eqref{eq:S*}. By Proposition~\ref{cor:canonicalMorphism},
\[
(Q_\alpha^L)^{**}(\Delta_0^*(\B)) = \Delta_0^*(\MM) \text{ and }
(Q_\alpha^L)^{**}(\Delta_0(\B)) = \Delta_0(\MM).
\]
Therefore, $\overline{\B}^{w^*} = \Delta_0(\MM)$. Moreover, since $\{0\}$ is a closed and open set, we have $0\notin \overline{\B\setminus \{0\}}^{w^*}$ and $\overline{\B\setminus \{0\}}^{w^*}\cup \{0\} = \overline{\B}^{w^*}$. Thus,
\[
\Delta(\MM) = \overline{\B}^{w^*}\setminus \{0\} = \overline{\B\setminus \{0\}}^{w^*}.\qedhere
\]
\end{proof}

The following result is an interesting analogue of \cite{WeaverBook2018}*{Corollary 7.27}. In particular it shows that the linear topological structure of Lipschitz-free spaces is completely determined by the purely algebraic structure of their duals.
\begin{Theorem}\label{thm:surprisignlyStrongResult}
Let $\NN$ be another metric space. Then there exists an algebra isomorphism between $\Lip_0(\MM)$ and $\Lip_0(\NN)$ if and only if $\B(\MM,\alpha)\simeq_\Lip \B(\NN,\alpha)$.
\end{Theorem}

\begin{proof}
Suppose first that $\B(\MM,\alpha)\simeq_\Lip \B(\NN,\alpha)$. Then it is clear that $\Lip_0(\B(\MM,\alpha))$ and $\Lip_0(\B(\NN,\alpha))$ are algebraically isomorphic with their standard pointwise product, which in turn implies that $\Lip_0(\MM)$ and $\Lip_0(\NN)$ are algebraically isomorphic with the product $\odot_\alpha$.

Conversely, suppose that there exists an algebra isomorphism
\[
T\colon\Lip_0(\MM)\to\Lip_0(\NN).
\]
By Proposition~\ref{p:algHomBdd}, $T$ is an isomorphism between the Banach spaces $\Lip_0(\MM)$ and $\Lip_0(\NN)$. Since $T$ preserves multiplication and its dual map $T^*$ is $w^*$-$w^*$ continuous, we claim that $T^*(\Delta_0^*(\NN)) = \Delta_0^*(\MM)$. Indeed, it suffices to check that for every $m\in \Delta_0^*(\NN)$ the map
\[
T^*(m)\colon \Lip_0(\MM)\to \Rea
\]
preserves multiplication and is $w^*$-continuous, which is routine using the properties of $T$. By Theorem~\ref{thm:w*EqualsNorm} we obtain that $\B(\MM,\alpha)$ and $\B(\NN,\alpha)$ are Lipschitz isomorphic.
\end{proof}
An immediate consequence of Theorems~\ref{thm:surprisignlyStrongResult} and~\ref{thm:boundedEquivToUnbounded} is the following.
\begin{Corollary}
If there is an algebra isomorphism between $\Lip_0(\MM)$ and $\Lip_0(\NN)$, then $\F(\MM)\simeq \F(\NN)$.
\end{Corollary}

Note also that a similar result holds as well for (nonlinear) lattice isomorphisms. This follows from \cite{CC11}*{Theorem 1}, where it is proved that two bounded metric spaces $\MM$ and $\NN$ are Lipschitz isomorphic if and only if there is a (nonlinear) lattice isomorphism betweeen $\Lip(\MM)$ and $\Lip(\NN)$.

\section{Simplifications of existing proofs using our construction}\label{Sec:6}
\noindent
There are several results on Lipschitz spaces and on Lipschitz free-spaces that work for unbounded metric spaces but whose proofs are much easier for bounded ones. This section is devoted to exhibiting that our methods permit to circumvent the technicalities that one encounters when proving some of this results for unbounded spaces. Our choice of the known results below is rather arbitrary and non exhaustive. We believe there are many more applications of our techniques in this direction and encourage the reader to further exploit them. In this section we deal with real Banach spaces only (the complex-scalar case is not considered here).
\subsection{Normal functionals}
The main result of \cite{AP20} consists of extending to any $\omega\in(\Lip_0(\MM))^*$ the equivalence between \ref{nC:1} and \ref{nC:2} in Lemma~\ref{lem:normalCharacters}, that is, in proving that every normal functional $\phi\in\F(\MM)^{**}$ is $w^*$-continuous. The proof carried out by the authors of \cite{AP20} is easy for bounded metric spaces but for unbounded metric spaces the authors need \cite{AP20}*{Lemma 4} and \cite{AP20}*{Lemma 5} (in particular a very deep analogue of ``Kalton's decomposition'') which reduces the result to investigating Lipschitz-free spaces over annuli.

Let $\alpha=\alpha^{(0)}$, so that $0$ is an isolated point of $\B(\MM,\alpha)$. Since by Proposition~\ref{cor:canonicalMorphism} the isomorphism $Q_\alpha^L$ defined as in \eqref{eq:S*} preserves normality and $w^*$-continuity, it suffices to prove the result for bounded metric spaces $\B$ with $d(0,x)=1$ for all $x\in\B\setminus\{0\}$. Thus, our construction would allow the authors to make the reduction to the investigation of Lipschitz-free spaces over annuli much easier. In order to appreciate this reduction let us briefly sketch the most important simplifications in this special case when compared with the much more technically involved proof from \cite{AP20}.

\begin{Theorem}
Let $\MM$ be a pointed metric space such that $d(x,0)=1$ for all $x\in\MM\setminus\{0\}$. A functional $\phi\in\F(\MM)^{**}$ is normal if and only if it is $w^*$-continuous.
\end{Theorem}

\begin{proof}[Sketch of the simplifications of the proof from \cite{AP20}]
We refer the reader to \cite{AP20}*{proof of Theorem 2}. Simplifications of the proof in this special case are the following:
\begin{itemize}[leftmargin=*]
\item We circumvent \cite{AP20}*{Lemmas 4 and 5}. Moreover, in this particular case we have $r=1=R$.
\item Our assumptions yield $\Vert f\Vert_\infty\le \Lip(f)$ for all $f\in\Lip_0(\MM)$ and $\Lip(\Ind_{\MM\setminus\{0\}})=1$, so instead of the functions $e$ and $e'$ we would use simply the function $\Ind_{\MM\setminus\{0\}}$ which would further simplify several computations.\qedhere
\end{itemize}
\end{proof}

\subsection{The Intersection Theorem and supports}
For Lipschitz-free spaces there is a well-defined notion of \emph{support} developed very recently in \cite{APP19}. Let us recall that the \emph{support} of $\gamma\in\F(\MM)$, denoted by $\supp(\gamma)$, is the smallest closed set $K$ such that $\gamma\in\F(K\cup\{0\})$. Note that if $L$ is the smallest closed set with $\gamma\in\F(L)$, then $\supp(\gamma)=L\setminus\{0\}$ in the case when $0$ is an isolated point of $L$, and $\supp(\gamma)=L$ otherwise. The existence of the support is ensured by the following theorem.

\begin{Theorem}[cf.\ \cites{AP19,APP19}] \label{thm:intersectionTheorem}
Let $\{K_i\colon i\in I\}$ be a family of closed subsets of $\MM$ with nonempty intersection. Then
\[
\cap_{i\in I} \F(K_i) = \F\left(\cap_{i\in I} K_i\right).
\]
\end{Theorem}

To properly understand Theorem~\ref{thm:intersectionTheorem}, called the Intersection Theorem, we must take into account that for every subset $\NN$ of a metric space $\MM$ there is a canonical isometric embedding of $\F(\NN)$ into $\F(\MM)$. This crucial property does not transfer to the case $p<1$ (see  Remark~\ref{rmk:Lj}).

The Intersection Theorem was proved for bounded metric spaces in \cite{AP19} and extended to its full generality more recently in \cite{APP19}. Our construction shows that the theorem for bounded spaces immediately implies the general case. Moreover, the notion of a support is preserved by our construction (see Proposition~\ref{prop:spt}). Recall that $\mu\colon\MM\rightarrow \B$ is the map from \eqref{eq:notation:b} and $P_\alpha^F\colon\F(\MM)\to\F(\B)$ is the isomorphism in Theorem~\ref{thm:boundedEquivToUnbounded}. 
For further reference we write down a lemma that we will use a couple of times.
\begin{Lemma}\label{lem:Transfersubspace}
Let $\MM$ be a metric space, and let $0\in\NN\subset \MM$. Then $P_\alpha^F(\F(\NN))=\F(\mu(\NN))$.
\end{Lemma}

\begin{Proposition}\label{prop:spt}
If the Intersection Theorem holds for bounded metric spaces, then it holds for all metric spaces. Moreover, for any metric space $\MM$ and $\gamma\in\F(\MM)$,
\[
P_\alpha^F(\supp(\gamma)\cup\{0\})=\supp(P_\alpha^F(\gamma))\cup \{0\}.
\]
\end{Proposition}
\begin{proof}
Let $\MM$ be an arbitrary metric space and let $\{K_i\colon i\in I\}$ be a family of closed subsets of $\MM$ with non-empty intersection. Without loss of generality we may assume that $0\in \bigcap_{i\in I}K_i$. By Lemma~\ref{lem:propertiesOfTransformation}~\ref{PoT:12}, each $\mu(K_i)$ is closed in $\B$, and so is the set $\mu(\cap_{i\in I} K_i)=\cap_{i\in I} \mu(K_i)$. By the Intersection Theorem for bounded spaces, 
\[
\cap_{i\in I} \F(\mu(K_i))=\F(\cap_{i\in I} \mu(K_i))=\F(\mu(\cap_{i\in I} K_i)).
\]
Applying Lemma~\ref{lem:Transfersubspace}, and using that $P_\alpha^L$ is an isomorphism, puts an end to the proof.
The  ``moreover'' part is shown similarly.
\end{proof}

\begin{Remark}
Our construction actually shows that the proof of the Intersection Theorem, which works for bounded metric spaces, also works for unbounded metric spaces: we need only Theorem~\ref{thm:keyPartInTheIntersectionTheorem} instead of \cite{AP19}*{Proposition 3.2}; the remainder of the proof is exactly the same as in the proof presented in \cite{AP19}*{proof of Theorem 3.3}.
\end{Remark}

\subsection{Compact reduction} Let $W$ be a subset  of a Lipschitz-free space $\F(\MM)$. The set  $W$ is said to be  \emph{tight} if for every $\varepsilon>0$ there is $K\subset\MM$ compact such that 
\[
W\subset \F(K)+\varepsilon B_{\F(\MM)}.
\]
Aliaga et al.\ proved in  \cite{ANPP2020}*{Theorem~2.3}  that every weakly precompact subset of $\F(\MM)$ is tight. Let us explain how our construction helps to simplify the proof of this result. If $W\subset \F(\MM)$ is weakly precompact then $P_\alpha^F(W)$ is a weakly precompact subset of $\F(\B(\MM))$. Combining \cite{ANPP2020}*{Proposition 3.3 and  Theorem 3.2} (which are stated for bounded metric spaces!) gives that $P_\alpha^F(W)$ is tight. By Lemma~\ref{lem:Transfersubspace} and Proposition~\ref{lem:propertiesOfTransformation}~\ref{PoT:11}, we infer that $W$ is tight.

\medskip

\noindent{\bf Acknowledgments.} The authors would like to thank  Nik Weaver for bringing to their attention  the connections of  the contents of this paper with previous developments in the theory, 
and Colin Petitjean
for his feedback on the applicability of our results to compact reduction in Lipschitz free spaces.


\end{document}